\documentclass{amsart}

\usepackage{amssymb,amsfonts, latexsym,amsmath,hhline,array,longtable}
\usepackage{pdfsync,color, comment,colortbl, mathrsfs,stmaryrd,cite,graphicx,yfonts}

\usepackage{ifpdf}
\ifpdf \usepackage[colorlinks=true, citecolor=blue, linkcolor=blue, urlcolor=blue]{hyperref} \fi

\newcommand{\cal}{\mathcal}

\newtheorem{formula}{}[section]
\newtheorem{definition}[formula]{Definition}
\newtheorem{corollary}[formula]{Corollary}
\newtheorem{remark}[formula]{Remark}
\newtheorem{lemma}[formula]{Lemma}
\newtheorem{theorem}[formula]{Theorem}
\newtheorem{proposition}[formula]{Proposition}
\newtheorem*{claim}{Claim}

\def\thrm{\begin{theorem}}
\def\thrml#1{\begin{theorem}\label{#1}}
\def\ethrm{\end{theorem}}
\def\prpstn{\begin{proposition}}
\def\prpstnl#1{\begin{proposition}\label{#1}}
\def\eprpstn{\end{proposition}}
\def\rmrk{\begin{remark}}
\def\rmrkl#1{\begin{remark}\label{#1}}
\def\ermrk{\end{remark}}
\def\dfntn{\begin{definition}}
\def\dfntnl#1{\begin{definition}\label{#1}}
\def\edfntn{\end{definition}}
\def\nmrt{\begin{enumerate}}
\def\enmrt{\end{enumerate}}
\def\tm#1{\item[{\rm (#1)}]}
\def\qtnl#1{\begin{equation}\label{#1}}
\def\eqtn{\end{equation}}
\def\lmm{\begin{lemma}}
\def\lmml#1{\begin{lemma}\label{#1}}
\def\elmm{\end{lemma}}
\def\crllr{\begin{corollary}}
\def\crllrl#1{\begin{corollary}\label{#1}}
\def\ecrllr{\end{corollary}}
\def\css{\begin{cases}}
\def\ecss{\end{cases}}
\def\prf{\begin{proof}}
\def\eprf{\end{proof}}
\def\clm{\begin{claim}}
\def\eclm{\end{claim}}

\def\cA{{\cal A}}

\def\cX{{\cal X}}

\def\mF{{\mathbb F}}

\def\mZ{{\mathbb Z}}

\def\fX{{\mathfrak X}}
\def\fY{{\mathfrak Y}}
\def\fZ{{\mathfrak Z}}

\def\sX{{\mathscr{X}}}

\DeclareMathOperator{\aut}{Aut}
\DeclareMathOperator{\alt}{Alt}
\DeclareMathOperator{\AGL}{AGL}

\DeclareMathOperator{\GL}{GL}

\DeclareMathOperator{\mon}{Mon}

\DeclareMathOperator{\orb}{Orb}

\DeclareMathOperator{\pr}{pr}
\DeclareMathOperator{\PSL}{PSL}
\DeclareMathOperator{\PGL}{PGL}
\DeclareMathOperator{\PGaL}{P\Gamma L}

\DeclareMathOperator{\rad}{rad}

\DeclareMathOperator{\Span}{Span}

\DeclareMathOperator{\sym}{Sym}
\DeclareMathOperator{\WL}{WL}

\def\mmod#1#2#3{#1=#2\ (\text{\rm mod}\hspace{2pt}#3)}

\def\qaq{\quad\text{and}\quad}

\def\und#1{{\underline{#1}}}

\begin{document}

\title{On circulant ternary coherent configurations of prime degree}

\author{Gang Chen}
\address{School of Mathematics and Statistics, Hainan University, Haikou 570228, China}
\email{chengangmath@mail.ccnu.edu.cn}
\author{Qing Ren}
\address{Guangxi Technological College of Machinery and Electricity, Nanning, China} \email{renqing@gxcme.edu.cn}
\author{Ilia Ponomarenko}
\address{School of Mathematical Sciences, Hebei Key Laboratory of Computational Mathematics
	and Applications, Hebei Normal University, Shijiazhuang 050024, P. R. China}
\address{Steklov Institute of Mathematics at St. Petersburg, Russia}
\email{inp@pdmi.ras.ru}
\thanks{The first author was supported by NSFC (project No.12371019). The third author was supported by the grant of The Natural Science Foundation of Hebei Province (project No. A2023205045).}
\date{}

\begin{abstract}
Ternary coherent configurations are, on the one hand, a special case of multidimensional coherent configurations introduced by L.~Babai (2016), and, on the other hand, a natural generalization of association schemes on triples introduced by D.~M.~Mesner and P.~Bhattacharya (1990).  A ternary coherent configuration $\fX$ is said to be circulant if the automorphism group~$\aut(\fX)$ of~$\fX$  has a regular cyclic  subgroup, and schurian if the classes of $\fX$  are the orbits of  the componentwise action of the group $\aut(\fX)$ on triples of points of~$\fX$. It is proved that any circulant ternary coherent configuration $\fX$ of prime degree~$p$ is schurian with the possible exception of the case when $\fX$ is an association schemes on triples and either $\aut(\fX)=\AGL_1(p)$ and $p=\pm 1\pmod{8}$, or  $\aut(\fX)\lvertneqq\AGL_1(p)$. 
\end{abstract}

\maketitle

\section{Introduction}

It is a common principle in algebraic combinatorics that objects with rich automorphism group are easier to analyze. In the most symmetric case, such objects can be described entirely in terms of relations of varying arity, on which the automorphism group acts transitively. However, this paradigm does not always apply: sometimes the object under study has small automorphism group. A classical example is given by the so-called Schur rings introduced by I.~Schur in 1933. As  H.~Wielandt later observed in~\cite{Wie1969}, Schur believed that all Schur rings should arise naturally from their automorphism groups. Wielandt’s counterexamples, presented in his monograph on permutation groups~\cite[Theorem~26.4]{Wie1969a}, disproved this expectation and led to the formulation of the Schurity problem for Schur rings: to determine which Schur rings arise from their automorphism groups in the sense envisioned by Schur.

In this paper, we address the Schurity problem for multidimensional \emph{coherent configurations}. These were introduced by L.~Babai~\cite{Babai2015c} (see also~\cite{AndresHelfgott2017}) in connection with the graph isomorphism problem, as a natural extension of the binary coherent configurations originally defined by D.~Higman in his study of permutation groups~\cite{Hig1970a}. The motivation for considering multidimensional coherent configurations is twofold. On the one hand, they arise naturally in the context of the graph isomorphism problem and the Weisfeiler--Leman method (see, e.g., \cite{CaiFI1992}). On the other hand, unlike in the binary case -- where the automorphism groups are precisely the $2$-closed groups in Wielandt’s terminology~\cite{Wie1969} -- the automorphism groups of multidimensional coherent configurations can be arbitrary permutation groups. 

By definition (see Section~\ref{180925f}), an \emph{$m$-ary coherent configuration} $\fX$ on a finite set~$\Omega$ is precisely the partition of the Cartesian power $\Omega^m$, constructed by the $m$-dimensional Weisfeiler--Leman (WL) algorithm; the number $|\Omega|$ is called the \emph{order} of~$\fX$. We say that $\fX$  is  \emph{schurian} if $\fX=\orb_m(G)$ is the set of all orbits of a group $G\le \sym(\Omega)$ with respect to the componentwise action on~$\Omega^m$. Nonschurian $m$-dimensional coherent configurations can serve as examples of graphs or higher-arity relational structures that are hard instances for isomorphism  testing using  the WL algorithm. At present, however, the only known hard examples are the graphs constructed by Cai, F\"urer, and Immerman~\cite{CaiFI1992} and their modifications.

An interesting special case of ternary coherent configurations -- \emph{association sche\-mes on triples }-- was introduced and studied in~\cite{MesnB1990}; in contrast to general case, the projection of an association scheme on triples to any two coordinates is a partition with exactly two classes: the diagonal of Cartesian square and its complement. The schurity of such ternary coherent configuration implies that its automorphism group is $2$-transitive. Schurian ternary coherent configurations corresponding to some $2$-transitive groups were further investigated in~\cite{BalmacedaB2025}. 

The nonschurian case appears considerably more intricate; here, of special interest  are  those nonschurian ternary coherent configurations whose projection onto the first two coordinates is a schurian binary coherent configuration. In this setting, known examples are largely limited to association schemes on triples described in~\cite{MesnB1990} and based on specific designs. The constructions presented  in~\cite{WuRP2025} suggest that the examples we are interested in can also be found among \emph{circulant} ternary coherent configurations, that is, those that admit a regular cyclic automorphism group. The case of circulant association schemes on triples was studied in~\cite{BhattacharyaP2021}. It turned out that even in this rather special class, it is difficult to describe those circulant association schemes on triples whose ``essential'' intersection numbers are~$1$, the so-called thin circulants in terminology of~\cite{BhattacharyaP2021}.

In the present paper, we restrict ourselves to circulant ternary coherent configurations $\fX$ of prime order. Our first main result shows that the most difficult case here arises if $\fX$ is an association scheme on triples. 

\thrml{060326a}
Any circulant ternary coherent configuration of prime order, other than  an association scheme on triples,  is schurian.
\ethrm

The proof of Theorem~\ref{060326a}, presented in Subsection~\ref{030426a}, is more or less direct. Indeed, let $\fX$ be a circulant ternary coherent configuration of prime order. The assumption that $\fX$ is not an association scheme on triples implies that the projection of $\fX$ to any two coordinates is a nontrivial cyclotomic association scheme of prime degree. In accordance with~\cite[Theorem~1.2(2)]{EvdP2003}, the latter scheme has a combinatorial base at most~$2$. This enables us to determine the classes of~$\fX$ explicitly. It turns out that each of them is an orbit of the group $\aut(\fX)$ acting on the triples, which means that $\fX$ is schurian.

\thrml{310825a}
Let $\fX$ be a (circulant) association scheme on triples, that  has prime order~$p$ and admits  a $2$-transitive automorphism group. Then $\fX$ is schurian unless $\aut(\fX)=\AGL_1(p)$ and $\mmod{p}{\pm 1}{8}$.
\ethrm

We believe that Theorem~\ref{310825a} is true without any assumption on the prime~$p$. At least, we have no counterexamples and, moreover, our proof can easily be extended to cover the case where $p-1=2^ar^b$, where $a\ge 1$ and $b\ge 0$ are integers.

We briefly outline the proof of Theorem~\ref{310825a}, which is presented in Subsection~\ref{030426b}. By the Burnside theorem on transitive groups of  prime degree, the group $\aut(\fX)$ 
is either $\AGL_1(p)$ or non-solvable (see also the explicit classification in~\cite{Jones2002}). In the non-solvable case, the stabilizer of two distinct points has at most two nonsingleton orbits, which rules out the existence of a nonschurian ternary coherent configuration with $2$-transitive automorphism group. 

The most difficult  case $\aut(\fX)=\AGL_1(p)$ is considered in Section~\ref{180925v} (Theorem ~\ref{310825g}). The proof there is based on the fact that a certain residue of the ternary coherent configuration~$\fX$ determines a Schur ring over a cyclic group. Such rings admit a more or less detailed classification~\cite{LeuM1996,EvdP2003}, explained also in Section ~\ref{180925v}. Using this classification and the assumption that~$2$ is a quadratic non-residue modulo~$p$, we can first prove that the Schur ring in question is either trivial or a group ring, and then deduce from this that the ternary coherent configuration~$\fX$ is schurian. The application of the classification exploits several number-theoretic properties of finite fields of prime order; we have collected them in Section ~\ref{100925a}.

As an immediate corollary of Theorems~\ref{060326a} and~\ref{310825a}, we obtain the following statement which shows where nonschurian circulant ternary coherent configurations of prime order can hide.

\crllrl{060326y}
Any circulant ternary coherent configuration $\fX$ of prime order is schurian unless $\fX$ is an  association scheme on triples such that  either $\aut(\fX)=\AGL_1(p)$ with $\mmod{p}{\pm 1}{8}$, or $\aut(\fX)\lvertneqq \AGL_1(p)$.
\ecrllr

All facts and concepts concerning multidimensional coherent configurations that are needed to understand the paper are presented in Section~\ref{180925f}.

\section{Multidimensional coherent configurations}\label{180925f}

\subsection{Notation}\label{230226a}
Throughout the paper, $\Omega$ denotes a finite set and $m\ge1$ an integer. Set $M=\{1,\dots,m\}$. For a tuple $x=(x_1,\dots,x_m)\in\Omega^m$, we write $x_i$ for its $i$th coordinate, $i\in M$. Denote by  $\rho(x)$ the equivalence relation on $M$ given by
$$
i\sim j \iff x_i=x_j.
$$
The number of classes of this equivalence relation is denoted by $\|\rho(x)\|$.

Denote by $\mon(M)$ the monoid of all maps from $M$ to itself (they act on the set of coordinates); for $\sigma\in\mon(M)$, we put
$$
x^\sigma=(x_{1^\sigma},\dots,x_{m^\sigma}).
$$
For $x\in\Omega^m$, $i\in M$, and $\alpha\in\Omega$, denote by $x_{i\leftarrow\alpha}$ the tuple obtained from $x$ by replacing the $i$th coordinate with $\alpha$. 

Let $I\subseteq M$ be a nonempty set of coordinate indices and write $\pr_I:\Omega^m\to\Omega^I$ for the projection onto the coordinates in~$I$. When $I=\{1,\ldots,k\}$, we abbreviate $\pr_I=\pr_k$. For $X\subseteq \Omega^m$ and $u\in \Omega^I$, we  set
$$
X_u=\{\pr_{M\setminus I}(x):\ x\in X\text{ and }\pr_I(x)=u\}.
$$
Clearly, if $J\subsetneq M$ is a nonempty subset disjoint with~$I$, then $(X_u)_v=X_{uv}$ for all $u\in \Omega^I$ and $v\in \Omega^J$, where $uv\in \Omega^{I\cup J}$ is the unique tuple for which $\pr_I(uv)=u$ and $\pr_J(uv)=v$.

We use standard notation: $\sym(\Omega)$ for the symmetric group on $\Omega$, $C_p$ for the cyclic group of order $p$, and $\AGL_1(p)=\AGL_1(\mF)$ for the $1$-dimensional linear affine group over a field $\mF=\mF_p$ with $p$ elements. 

\subsection{Definition}
A partition $\fX$ of  $\Omega^m$ is called an \emph{$m$-ary coherent configuration} on~$\Omega$  if the following conditions are satisfied for all $X\in \fX$:
\nmrt
\tm{C1} $\rho(x)=: \rho(X)$ does not depend on the choice of $x\in X$,
\tm{C2} $X^\sigma\in \fX$  for all $\sigma\in\mon(M)$,
\tm{C3} for any $X_1,\ldots,X_m\in \fX$, the number 	
$$
n_{X_1,\dots,X_m}^X=|\{\alpha\in\Omega:\;x_{i\leftarrow\alpha}\in X_i\ \text{for all } i\in M\}|
$$
does not depend on the choice of $x\in X$.
\enmrt
For $m=2$, this definition reproduces the classical notion of a coherent configuration~\cite{CP2019}.  Let $m=3$ and for any equivalence relation $\rho\subseteq M\times M$ with at most two classes, there is exactly one  $X\in\fX$ for which $\rho(X)=\rho$. Then the $3$-ary coherent configuration $\fX$ is called an \emph{association scheme on triples}, see~\cite{MesnB1990}.

\subsection{Automorphisms and schurity}
Let $\fX$ be an $m$-ary coherent configuration on~$\Omega$. 
The \emph{automorphism group} $\aut(\fX)$ is the subgroup of $\sym(\Omega)$ consisting of all permutations~$f$ such that $X^f=X$ for every $X\in\fX$, where the action is defined 
componentwise on~$\Omega^m$:
$$
x^f=(x_1,\dots,x_m)^f = (x_1^f,\dots,x_m^f)
$$
for all $x\in\Omega^m$. Thus the group $\aut(\fX)$ acts on $\Omega$ and leaves each class of the partition $\fX$ fixed. It is not difficult to verify that $(x^\sigma)^f=(x^f)^\sigma$ for all~$x\in\Omega^m$, $\sigma\in\mon(M)$, and $f\in\sym(\Omega)$; in particular, $(X^\sigma)^f=(X^f)^\sigma=X^\sigma$ for all $X\in \fX$ and $f\in\aut(\fX)$. 


Let $G\le\sym(\Omega)$ be a permutation group. If $G\le\aut(\fX)$, then we say that $\fX$ is \emph{invariant} with respect to $G$.  Any set of the form $x^G=\{x^f:\ f\in G\}$  with $x\in\Omega^m$ is called an \emph{$m$-orbit} of $G$.  The partition 
$$
\orb_m(G)=\{x^G:\ x\in\Omega^m\}
$$ 
of the Cartesian power $\Omega^m$ into the $m$-orbits of~$G$  is an $m$-ary coherent configuration.  A coherent configuration $\fX$ is called \emph{schurian} if $\fX=\orb_m(G)$ for some $G\le\sym(\Omega)$. Note that in general, the group~$G$ with this property is not a unique one.

\subsection{Galois correspondence}
The $m$-ary coherent configurations on $\Omega$ are partially ordered by refinement: we write
$\fX\le\fX'$ if every class of $\fX$ is a union of classes of~$\fX'$.  It is straightforward to verify the monotonicity relations
\qtnl{110326y}
\fX\le\fX'\ \Longrightarrow\ \aut(\fX)\ge\aut(\fX'),\qquad
G\le G'\ \Longrightarrow\ \orb_m(G)\ge\orb_m(G'),
\eqtn
for partitions $\fX,\fX'$ of~$\Omega^m$ and subgroups $G,G'\le\sym(\Omega)$. Moreover, the two maps
$$
\fX\longmapsto \aut(\fX),\qquad G\longmapsto \orb_m(G)
$$
form a Galois connection between the lattice of $m$-ary coherent configurations (ordered by refinement) and the lattice of permutation subgroups of~$\sym(\Omega)$ (ordered by inclusion). In particular, for every $m$-ary coherent configuration $\fX$ and every group $G$, we have the following identities:
$$
\aut(\orb_m(\aut(\fX)))=\aut(\fX),\qquad
\orb_m(\aut(\orb_m(G)))=\orb_m(G).
$$
The Galois closed objects under this correspondence are the schurian $m$-ary coherent configurations and $m$-closed permutation groups, see~\cite{Wie1969}.

\subsection{Projections and residues}\label{170925r}

Let $I\subseteq M$ be a nonempty subset and $y\in\Omega^{M\setminus I}$. In accordance with~\cite{Babai2015c,AndresHelfgott2017}, the sets 
$$
\pr_I\fX=\{\pr_I X:\ X\in\fX\},\qquad
\fX_y=\{X_y:\ X\in\fX,\ X_y\neq\varnothing\}
$$
are $k$-ary coherent configurations on~$\Omega$ , where $k=|I|$; they are called the \emph{projection} of $\fX$ to~$I$, and the \emph{residue} of $\fX$ with respect to~$y$, respectively. Clearly,
$$
\fX_y\ge \pr_I\fX. 
$$
The lemma below immediately follows from the definitions.

\lmml{070326a}
A ternary coherent configuration $\fX$ is an association scheme on triples if and only if $\pr_2\fX$ is a trivial coherent configuration.\footnote{A binary coherent configuration is said to be trivial if it consists of  at most~$2$ classes.}
\elmm

Taking projections and residues respects the automorphism groups in the following sense. First, given $I\subseteq M$ and $y\in\Omega^{M\setminus I}$,  one has the inclusions
\qtnl{170925b}
\aut(\pr_I\fX)\ge \aut (\fX),\qquad \aut(\fX_y)\ge \aut(\fX)_y.
\eqtn
Second,  if  $\fX=\orb_m(G)$ for some $G\le\sym(\Omega)$, then 
$$
\pr_k\fX=\orb_k(G),\qquad\fX_y=\orb_k(G_y),
$$ 
where  $k=|I|$ and $G_y$ is the pointwise stabilizer of the entries of~$y$ in the group~$G$.

The projections $\pr_1 X$, $X\in\fX$, are called \emph{fibers} of the $m$-ary coherent configuration~$\fX$. It is not hard to verify that if $\Delta$ is a fiber of~$\fX$, then the partition of~$\Delta^m$ into the nonempty classes $X\cap\Delta^m$, $X\in\fX$, is an $m$-ary coherent configuration on~$\Delta$; we denote it by $\fX^\Delta$  and call the \emph{restriction} of~$\fX$ to~$\Delta$.  When $m=2$, these definitions are compatible with that for binary coherent configurations. 

\section{Schur partitions}\label{180925v}
Let $G$ be a group. Following~\cite[Definition~4.1]{Muzychuk2004}, a partition $\Pi$ of~$G$ is called a \emph{Schur partition} if $\{1\}\in \Pi$, the set $X^{-1}=\{x^{-1}:\ x\in X\}$ is a class of~$\Pi$ for all $X\in\Pi$, and the $\mZ$-module
\qtnl{120925t}
\cA=\Span_\mZ\{\und{X}:\ X\in \Pi\}
\eqtn
is a subring of the group ring $\mZ G$, where $\und{X}\in \mZ G$ is the formal sum of the elements of $X$. The subrings of the form~\eqref{120925t} corresponding to the Schur partitions are exactly the Schur rings over the group~$G$. The statement below is an immediate consequence of~\cite[Corollary~2.4.18]{CP2019}.

\prpstnl{190925s}
Let $\fX$ be a binary coherent configuration on the elements of a group~$G$. Assume that $\fX$ is invariant  with respect to the group induced by right multiplications of~$G$. Then the residue $\fX_{(1)}$ with respect to the tuple $(1)$ is a Schur partition of~$G$.
\eprpstn

Two obvious examples of the Schur partition $\Pi$ are the \emph{discrete} and \emph{trivial} partitions of~$G$: in the first case, $|\Pi|=|G|$ and the classes of $\Pi$ are singletons, whereas in the second case, $|\Pi|\le 2$ and the classes of $\Pi$ are $\{1\}$ and $G\setminus\{1\}$ (if $G$ is not the identity group).  For any subgroup $K\le\aut(G)$, the partition $\orb(K,G)$ of $G$ into the orbits of ~$K$ is a Schur partition; it is called \emph{orbit} (or \emph{cyclotomic}) partition.

The Schur rings over cyclic group were completely classified in~\cite{LeuM1996}. In Section~\ref{120925h}, we will need  a refinement of this  classification, given in~\cite{EvdP2003}. We formulate it  in terms of Schur partitions, not  in terms of Schur rings as was done there. Below a subset~$X$ of a cyclic group $G$ is said to be \emph{highest} if $X$ contains a generator of~$G$;  the subgroup of all~$y\in G$ such that $yX=Xy=X$ is called the
\emph{radical}  of $X$ and is denoted by~$\rad(X)$. 

\begin{lemma}\label{030825c}
Let $\Pi$ be a Schur partition of a cyclic group~$G$.  Then exactly one of the following alternatives holds.
\nmrt
\tm{a} $\rad(X)\ne \{1\}$ for each highest class $X\in\Pi$.
\item[(b)] There exist a direct product decomposition $G=G_1\times\cdots\times G_k$ $(k\ge1)$ and a Schur partition $\Pi_i$ of the group $G_i$, $1\le i\le k$, such that
\nmrt
\tm{b1} $\Pi_1,\ldots,\Pi_k$ are trivial partitions except for at most one being  orbit;
\tm{b2} $\Pi=\{X_1\times\cdots\times X_k:\ X_1\in\Pi_1,\ldots,X_k\in\Pi_k\}$.
\enmrt
\enmrt
\end{lemma}


\prf
The statement is essentially a reformulation of the known classification of Schur rings over cyclic groups (see \cite{LeuM1996, EvdP2003}). Denote by $\cA$ the Schur ring associated with the partition~$\Pi$. By \cite[Corollary~6.4]{EvdP2003},  either the radical of $\cA$ is nontrivial, or~$\cA$ is a tensor product of a normal Schur ring with trivial radical and Schur rings of rank~$2$. The first possibility is equivalent to statement~(a): the radical of~$\cA$ is just the $\rad(X)$ for any highest class of~$\Pi$. The second possibility is equivalent to statement~(b): the decomposition of~$\cA$ in the tensor product means the existence of the direct product decomposition~$G=G_1\times\cdots\times G_k$ satisfying condition~(b2), the Schur partition associated with  normal Schur ring is always orbit  \cite[Theorem~6.1]{EvdP2003}, and the Schur rings of rank~$2$ are associated with trivial Schur partitions.
\eprf

\section{Arithmetic properties of finite fields of prime order}\label{100925a}
In this section, we establish two number-theoretic facts about finite fields of prime order. These auxiliary results will be used in the classification arguments of Section~5, especially in the case of circulant ternary coherent configurations invariant with respect to one-dimensional affine linear group of a prime field. 

\lmml{300825c}
Let $p>2$ be a prime. Then there exists a primitive root $x\in \mF_p$ such that $1-x$ is also a primitive root.
\elmm
\prf
It was proved in~\cite{CohenST2015} that for any prime power $q > 61$ and for arbitrary nonzero $a, b, c\in \mF_q$, there is always one representation of the form $a = bx+cy$, where $x$ and $y$ are primitive roots of $\mF_q$. Taking $q=p$ and $a=b=c=1$, we get the required result for all $p>61$. For the remaining  primes, the statement can be verified by a straightforward computation.
\eprf

Let $p$ be a prime and $\mF=\mF_p$. Let $X\subset \mF^\times$  be a nonempty subset not containing the element $1=1_\mF$. We say that $X$ is of \emph{group type} if the union  $X\cup\{1\}$ is a multiplicative subgroup of~$\mF$. Clearly, $X$ is of group type only if $\sum_{x\in X}x=-1$.

\lmml{300825b}
In the above notation, the following statements hold:
\nmrt
\tm{1} if $\rad(X)\ne\{1\}$, then $\rad(1-X)=\{1\}$,
\tm{2} if $X$ is of group type, then $1-X$ is of group  type only if $X=\mF^\times\setminus\{1\}$.
\enmrt
\elmm
\prf
Obviously,
\qtnl{1052c}
\sum_{x\in X}x+\sum_{y\in 1-X}y=|X|\cdot 1. 
\eqtn
Also, if $\rad(X)\ne 1$, then $X$ is a union of full cosets of $\rad(X)$ in $\mF^{\times}$. As the sum of elements in each coset is equal to $0$, the sum of elements in $X$ is $0$. 

Now if $\rad(1-X)\ne 1$, then by equality \eqref{1052c}, we have $|X|\cdot 1=0$ in $\mF$. This is impossible since $1\le |X|\le p-1$.  This proves statement (1). 

If both $X$ and $1-X$ are of group type, then by \eqref{1052c} we have $-2=|X|\cdot 1$ in $\mF$.  This yields that  $|X|+2=p$. Thus, $|X|=p-2$, which proves statement (2). 
\eprf

\section{Ternary coherent configurations invariant with respect to AGL$_1(p)$}\label{120925h}
The aim of this section is to prove Theorem~\ref{310825g} below. It represents a special case of Theorem~\ref{310825a} and, in a sense, forms a key part of the entire argument.

\thrml{310825g}
Let $\fX$ be a ternary coherent configuration on  a  field~$\mF=\mF_p$ of prime order $\mmod{p}{\pm 3}{8}$. Assume that $\fX$ is invariant with respect to the  group~$\AGL_1(\mF)$. Then 
$$
\fX=\orb_3(G),
$$ 
where $G=\sym(\mF)$ or~$\AGL_1(\mF)$.
\ethrm

The proof of Theorem~\ref{310825g} is given at the end of the section. In what follows, we fix  a ternary coherent configuration $\fX$ invariant with respect to the group~$\AGL_1(\mF)$.  Denote by~$\fX_{(0,1)}$ the residue of $\fX$ with respect to the pair $(0,1)\in\mF\times\mF$. 
It is easily seen that $\fX_{(0,1)}$ is a partition of~$\mF$ that contains two singleton classes, $\{0\}$ and~$\{1\}$.  We focus on the induced  partition
\qtnl{130925a}
\Pi =\{X\in \fX_{(0,1)}:\ X\ne \{0\}\}
\eqtn
of the multiplicative group~$\mF^\times$. Clearly, $\{1\}$ is a class of~$\Pi$. 

\lmml{020925a}
The partition $\Pi$ is a Schur partition of the group~$\mF^\times$. Moreover, if $X\in\Pi$ and $1\notin X$, then $1-X\in\Pi$.
\elmm
\prf
The residue $\fX_0=\fX_{(0)}$ is a binary coherent configuration. By the assumption, it is invariant with respect to the stabilizer of~$0$ in the group~$\AGL_1(\mF)$. Moreover, $\fX_0$ has two fibers: $\{0\}$ and $\mF^\times$. Therefore, the restriction of $\fX_0$ to the fiber~$\mF^\times$ is invariant with respect to  right multiplications of the group~$\mF^\times=\GL_1(\mF)$. By Proposition~\ref{190925s}, it follows that the restriction of the residue
$$
\fX_{(0,1)}=(\fX_{(0)})_{(1)}
$$
to $\mF^\times$ is a Schur partition of~$\mF^\times$.  Thus $\Pi$ is indeed a Schur partition.

To prove the second claim, take $X\in\Pi$. Then there exists a class $\sX\in\fX$ such that  its residue $\sX_{(0,1)}$ with respect to~$(0,1)$ is equal to~$X$. Let $\sigma=(1,\,2)$ be the transposition of~$\sym(3)$.  Then $\sX^\sigma$ is also a class of~$\fX$, and it contains all triples~$(1,0,x)$ with $x\in X$. 

Now consider the map $f:x\mapsto 1-x$ on $\mF$. Since $f\in\AGL_1(\mF)$ and  $\fX$ is invariant with respect to~$\AGL_1(\mF)$, we have 
$$
(\sX^\sigma)^f=\sX^\sigma.
$$ 
Therefore, $\sX^\sigma$ contains all triples 
$$
(0,1,1-x)=(1,0,x)^f, \qquad x\in X.
$$ 
This shows that $1-X\subseteq(\sX^\sigma)_{(0,1)}$. If the inclusion were strict, then by the same argument, we could enlarge~$X$ within~$\Pi$, contradicting the definition of~$\Pi$.  
Hence, $1-X=(\sX^\sigma)_{(0,1)}\in\Pi$.
\eprf

By Lemma~\ref{020925a}, $\Pi$ is a Schur partition over the cyclic group $\mF^\times$, that is invariant with respect to the permutation $\tau\in\sym(\Pi)$  taking each class $X\ne\{1\}$ to  $1-X$ and leaving the class $\{1\}$ fixed.  

\lmml{120925a}
The Schur partition $\Pi$  is discrete or trivial. 
\elmm
\prf 
By Lemma~\ref{030825c}, the partition $\Pi$ satisfies one of the alternatives (a) or (b) from that lemma. First, assume that $\rad(X)\ne\{1\}$ for each highest class~$X\in \Pi$.  

By Lemma~\ref{300825c}, there exist generators $x,y$ of the group~$\mF^\times$ such that $y=1-x$. Let $X,Y\in\Pi$ be the highest classes of~$\Pi$ that contain $x$ and $y$, respectively. Then 
$$
\rad(X)\ne\{1\}\ne\rad(Y).
$$ 
On the other hand, $1-X=X^\tau=Y$, which contradicts Lemma~\ref{300825b}(1). Thus case~(a) in Lemma~\ref{030825c} is impossible.

To deal with case~(b), we need the following claim in which (as well as in the proof of Lemma~\ref{020925a}), we do not use the assumption on the prime~$p$ in Theorem~\ref{310825g}.

\clm
The partition $\Pi$ is discrete if it has a singleton class different from $\{1\}$. 
\eclm
\prf
Denote by $H$ the union of all singleton classes of $\Pi$. Since the product of any two singleton classes of a Schur partition is also a class of this partition, the set~$H$ is a subgroup of the group $\mF^\times$. Hence the set $X=H\setminus\{1\}$ is of group type. On the other hand, the permutation $\tau$ permutes the singleton classes and fix the class~$\{1\}$. Therefore, $X^\tau=X$. By Lemma~\ref{300825b}(2), this implies that $X=\mF^\times\setminus\{1\}$. Consequently,  the partition~$\Pi$ is discrete.
\eprf

Now let $\mF^\times=G_1\times\cdots\times G_k$ $(k\ge1)$ and $\Pi_i$  the Schur partition of the group~$G_i$, $1\le i\le k$, such that conditions (b1) and (b2) of Lemma~\ref{030825c} are satisfied. Since the group $\mF^\times$ is cyclic, each $G_i$ is a product of some Sylow subgroups of $\mF^\times$. Denote by~$i$ the index for which the group $G_i$ is of even order. Then $G_i$ contains the Sylow $2$-subgroup and hence $-1\in G_i$. If the partition $\Pi_i$ is an orbit one,  then obviously the singleton $\{-1\}$ is a class of~$\Pi_i$   and hence is a class of~$\Pi$. But then the partition~$\Pi$ is discrete by the Claim, and we are done.

Now let the partition $\Pi_i$ be trivial, i.e., consist of exactly two classes, $\{1\}$ and~$X$, the union of which is~$G_i$. Put $Y=1-X=X^\tau$. Since $-1\in X$, we have $2=1-(-1)\in Y$. Furthermore,
\qtnl{210226b}
 Y=Y_1\times\cdots \times Y_k
\eqtn
for some classes $Y_1\in\Pi_1,\ldots,Y_k\in\Pi_k$. At this point, we make use of the assumption on the prime~$p$ to prove that for the chosen~$i$,
\qtnl{210226a}
Y_i=X.
\eqtn
Indeed, denote by $P$ the Sylow $2$-subgroup of the (cyclic) group $\mF^\times$. By the choice of~$i$, we have $P\le G_i$. Furthermore, the element $2$ can uniquely be written as  the product of an element $y\in P$ and an element  of the complement to~$P$ in~$\mF^\times$. On the other hand, it is well known that $\mmod{p}{\pm 3}{8}$ if and only if the element~$2$ is a quadratic non-residue modulo~$p$. Then $2$ is not a square in~$\mF$. Consequently,  $y$ is a generator of~$P$. Now since $2\in Y$, 
formula~\eqref{210226b} implies that $y\in Y_i$. Thus, $Y_i\ne\{1\}$ and equality~\eqref{210226a} follows.

By formulas \eqref{210226b} and \eqref{210226a}, we immediately obtain
$$
|X|=|1-X|=|Y|=|Y_i|\cdot\prod_{j\ne i}|Y_j|=|X|\cdot\prod_{j\ne i}|Y_j|.
$$
Consequently, $|Y_j|=1$ for all $j\ne i$. Now if $Y_j=\{1\}$ for all these~$j$, then $X=Y=1-X$ is of group type, and we conclude by Lemma~\ref{300825b}(2) that $X=\mF^\times\setminus\{1\}$, i.e., the partition~$\Pi$ is trivial. 

To complete the proof, let $j\ne i$ be an index such that $Y_j\ne\{1\}$. Since $|Y_j|=1$, the product $\{1\}\times\ldots\times\{1\}\times Y_j\times \{1\}\times\ldots\times\{1\}$  is a singleton class of~$\Pi$, different from~$\{1\}$. By the Claim, this implies that the partition~$\Pi$ is discrete.
\eprf

{\bf Proof of Theorem~\ref{310825g}.} Let $\fX^*$ be the set of all classes $\sX\in\fX$ such that  $\|\rho(\sX)\|=3$, i.e., the entries of each triple belonging to~$\sX$ are pairwise distinct.  The following statement is straightforward.

\lmml{230226d}
Let $\fX$ be a ternary coherent configuration invariant with respect to a $2$-transitive group~$G$. Then $\fX\setminus\fX^*\subseteq\orb_3(G)$. 
\elmm

The group $G=\AGL_1(\mF)$ is $2$-transitive. Hence by Lemma~\ref{230226d}, it suffices to verify that all classes $\sX\in\fX^*$  are $3$-orbits of either the symmetric group $\sym(\mF)$ or the group~$G$. By Lemma~\ref{120925a}, for each class $\sX\in\fX^*$, exactly one of the following statements holds:
$$
|\sX_{(0,1)}|=p-2 \qquad\text{or}\qquad |\sX_{(0,1)}|=1.
$$

In the former case, $\sX_{(0,1)}=\mF^\times\setminus\{1\}=:X$. Since $\fX$ is invariant with respect to~$G$, this implies that
$$
\sX\supseteq\bigcup_{f\in G}(\{0\}\times\{1\}\times X)^f=\bigcup_{f\in G}\{0^f\}\times\{1^f\}\times X^f=\sX^*,
$$
where $\sX^*$ is the union of all classes in~$\fX^*$. It follows that $|\fX^*|=1$ and hence $\fX=\orb_3(\sym(\mF))$.  

To complete the proof, assume that $|\sX_{(0,1)}|=1$ for all classes $\sX\in\fX^*$. It immediately follows that 
$$
|\sX|=p(p-1)=|G|
$$ 
for all such~$\sX$. On the other hand, since $\fX$ is invariant with respect to the group~$G$,  this equality holds only if $\sX$ is a $3$-orbit of $G$. Thus, in our case, $\fX^*\subseteq\orb_3(G)$, as required.

\section{Proof of Theorems~\ref{060326a} and~\ref{310825a}}\label{180925a}
\subsection{Proof of Theorem~\ref{060326a}} \label{030426a}
First we recall some relevant facts on cyclotomic schemes. Let $\mF$ be a finite field and let $G\le\AGL_1(\mF)$ be a group of  all permutations 
\qtnl{110326b}
x\mapsto ax+b,\ x\in\mF,
\eqtn
where $a$ belongs to some fixed subgroup~$K\le \mF^\times$ and $b\in\mF$. The binary coherent configuration $\orb_2(G)$ is called a \emph{cyclotomic scheme} over~$\mF$, associated with~$K$. 

\lmml{110326a}
Assume that the field $\mF$ is prime and $|\orb_2(G)|>2$. Let $\cX$ be the minimal coherent configuration on $\mF$ such that $\cX\ge\orb_2(G)$ and $\{(0,0)\}\in\cX$.\footnote{In terms of~\cite{CP2019}, $\cX$ is a one-point extension of the scheme~$\orb_2(G)$.} Then the residue $\cX_y$ is discrete for each $1$-tuple~$y\ne(0)$.
\elmm
\prf
By \cite[Theorem~1.2(2)]{EvdP2003}, we have $\cX=\orb_2(K)$. Since the identity is the only permutation of the form~\eqref{110326b} that belongs to~$K$ and leaves the entry $y_1$ fixed, we are done.
\eprf

Now, let  $\fX$ be a circulant ternary coherent configuration of prime order. Its  projection $\pr_2\fX$ is a  coherent configuration of prime order. Moreover, the group $\aut(\pr_2\fX)$ is  transitive. By \cite[Theorem~4.5.1]{CP2019}, any such coherent configuration  is a cyclotomic scheme over prime field. Thus,
\qtnl{080326t}
\pr_2\fX=\orb_2(G),
\eqtn
for some transitive group $G\le\AGL_1(p)$. Moreover, by the theorem hypothesis, we may assume that~$\fX$ is not an association scheme on triples. By Lemma~\ref{070326a}, this implies that  the cyclotomic scheme~$\pr_2\fX$ is nontrivial, i.e., $|\pr_2\fX|>2$. 

Let $\fY$ be the minimal ternary coherent configuration such that $\pr_2 \fY=\pr_2\fX$; in the notation  of \cite{ChenRP2025a}, we have $\fY=\WL_3(\pr_2\fX)$. Then 
\qtnl{110326h}
\fY\le\fX\qaq \fY\le\orb_3(G),
\eqtn
where the second inclusion follows from~\eqref{080326t}. The first inclusion shows also that $\aut(\fY)\ge\aut(\fX)$, see~\eqref{110326y}.

\lmml{110326c}
For any $2$-tuple $x$ of distinct points, the residue $\fY_x$ is discrete.
\elmm
\prf
The group $\aut(\fY)\ge\aut(\fX)$ is transitive. Hence without loss of generality, we may assume that $x_1=0$. Now we make use of the fact proved in~\cite[Lemma~3.1]{ChenRP2025a}, namely
$$
\fY_{(0)}\ge (\pr_2\fY)_0,
$$
where the left-hand side is the residue of $\fY$ with respect to the $1$-tuple $(0)$ and the right-hand side is the minimal coherent configuration $\cX$ such that $\cX\ge \pr_2\fY$ and $\{(0,0)\}\in\cX$. Since the coherent configuration $\pr_2\fY=\pr_2\fX=\orb_2(G)$ is nontrivial, Lemma~\ref{110326a} implies that the residue $\cX_y$ with respect to the $1$-tuple $y=(x_2)$ is discrete. Thus the residue
$$
\fY_x=(\fY_{(0)})_{(x_2)}\ge ((\pr_2\fY)_0)_y=\cX_y
$$
is also discrete, as required.
\eprf

The conclusion of Lemma~\ref{110326c} implies that if $\fZ\ge\fY$ is a ternary coherent configuration such that $\pr_2\fZ=\pr_2\fY$, then $\fZ=\fY$. In view of~\eqref{110326h}, we can  apply this observation to both $\fZ=\fX$ and $\fZ=\orb_3(G)$ to obtain
$$
\fX=\fY=\orb_3(G),
$$ 
which completes the proof.

\subsection{Proof of Theorem~\ref{310825a}}\label{030426b}
Let $\fX$ be a ternary coherent configuration invariant with respect to a regular cyclic group of prime order~$p$, and let $G=\aut(\fX)$.  According to the classification of permutation groups containing a regular cyclic subgroup \cite[Theorem~3]{Jones2002},  the group~$G$ falls into one of the following cases:
\nmrt
\tm{1} $C_p\le  G\le  \AGL_1(p)$, 
\tm{2} $G=\sym(p)$ or $G=\alt(p)$, where $p\ge 3$,
\tm{3} $\PGL_d(q)\le G\le\PGaL_d(q)$, where $p=(q^d-1)/(q-1)$ for some $d\ge 2$,
\tm{4} $G=\PSL_2(11)$, $M_{11}$, or $M_{23}$, where $p=11,11$, or~$23$, respectively.
\enmrt

Assume that $G$ is $2$-transitive. In case~(1), this implies that $G=\AGL_1(p)$, and the required statement follows from Theorem~\ref{310825g}.  Suppose that $G$ falls into cases~(2), (3), or~(4).  Let $\fX^*$ be the set of classes of~$\fX$, defined in the proof of Theorem~\ref{310825g}, and let $\orb_3(G)^*$ be defined in the same way for the ternary coherent configuration $\orb_3(G)$. If now $|\orb_3(G)^*|=1$, i.e., the group $G$ is $3$-transitive, then the invariance of $\fX$ with respect to~$G$ implies that $\fX=\orb_3(\sym(p))$. This rules out case~(2), case~(3) with $d=2$, and the groups $M_{11},M_{23}$ in case~(4) (see~\cite[Table~7.4]{Cam1999}).  In the remaining cases, we have
\qtnl{230226l}
|\orb_3(G)^*|=2.
\eqtn
In case (4) and $G=\PSL_2(11)$, this is verified by direct calculation. In case~(3) and $d\ge 3$, the 
group $\PGL_d(q)$ acts $3$-transitively on each projective line (see~\cite[Table~7.4]{Cam1999}) and transitively on the triples of noncollinear projective points. It follows that the  action of $G$ on the projective space PG$(d-1,q)$ yields exactly two orbits on triples with pairwise distinct entries: the triples of points lying on one projective line, and the triples not lying on the same line.  

It remains to verify that formula~\eqref{230226l} implies $\fX=\orb_3(G)$.  Indeed from~\eqref{230226l}, we obtain   
$$
2=|\orb_3(G)^*|\;\ge\;|\fX^*|\;\ge \;|\orb_3(\sym(p))^*|=1.
$$
Consequently, either $|\fX^*|=1$ or $|\fX^*|=2$. In the former case, we obviously have $\fX=\orb_3(\sym(p))$.  In the last case, $\fX=\orb_3(G)$ by Lemma~\ref{230226d}.

\bibliographystyle{amsplain}

\providecommand{\href}[2]{#2}

\end{document}